\documentclass[10pt, oneside,reqno]{amsart}
\usepackage{amsmath, amssymb, amsfonts}
\usepackage{amsthm}
\usepackage{hyperref}
\theoremstyle{definition}

\numberwithin{equation}{section}
\newtheorem{theorem}{Theorem}[section]
\newtheorem{proposition}[theorem]{Proposition}
\newtheorem{definition}[theorem]{Definition}

\newtheorem{example}[theorem]{Example}

\begin{document}
\author{Kamran Alam Khan}
\address{Department of Mathematics \\
 Ram Lubhai Sahani Govt. Mahila Degree College \\
 Pilibhit (U.P.)-INDIA}
\email{kamran12341@yahoo.com}
\title{Banach and Suzuki-type fixed point theorems in Generalized $n$-metric spaces with an application}
\pagestyle{myheadings}
\thispagestyle{empty}
\begin{abstract}
\noindent
  Mustafa and Sims \cite{MU2} introduced the notion of $G$-metric as a possible  generalization  of usual notion of a metric space. The author generalized the notion of G-metric to more than three variables and introduced the concept of Generalized $n$-metric spaces \cite{KHA2}. In this paper, We prove Banach fixed point theorem and a Suzuki-type fixed point theorem in Generalized $n$-metric spaces. We also discuss applications to certain functional equations arising in dynamic programming.
\end{abstract}

\maketitle

\section{INTRODUCTION}
There are many ways to generalize the notion of a metric space. 2-metric space (\cite{GA1} ,\cite{GA2}), $D$-metric space \cite{DHA} and $G$-metric space \cite{MU2} are the most familiar generalizations. The author  generalized the notion of $G$-metric to more than three variables and introduced the concept of $K$-metric \cite{KHA1} and the generalized $n$-metric \cite{KHA2} . In this paper, we prove the Banach fixed point theorem and the Suzuki-type fixed point theorem in the framework of generalized $n$-metric space. We also discuss applications to certain functional equations arising in dynamic programming.
\begin{definition} \cite{KHA2}
Let $X$ be a non-empty set, and $\mathbb{R}^+$ denote the set of non-negative real numbers. Let $G_n\colon X^n \to \mathbb{R}^+$, $(n\ge 3)$ be a function satisfying the following properties:
\begin{itemize}
\item [[G 1]] $G_n(x_1,x_2,...,x_n)=0$ if $x_1=x_2=\dots =x_n$,
\item [[G 2]] $G_n(x_1,x_1,...,x_1,x_2)>0$ for all $x_1$, $x_2 \in X$ with $x_1\neq x_2$,
\item [[G 3]] $G_n(x_1,x_1,...,x_1,x_2)\leq G_n(x_1,x_2,...,x_n)$ for all $x_1$, $x_2,... ,x_n \in X$ with the condition that any two of the points $x_2,\cdots ,x_n$ are distinct,
\item [[G 4]] $G_n(x_1,x_2,...,x_n)=G_n(x_{\pi(1)},x_{\pi (2)},...,x_{\pi(n)})$, for all $x_1$, $x_2,...,x_n \in X$ and every permutation $\pi$ of $\{1,2,...n\}$,  
\item [[G 5]] $G_n(x_1,x_2,...,x_n)\le G_n(x_1,x_{n+1},...,x_{n+1})+G_n(x_{n+1},x_2,...,x_n)$ for all\\ $x_1$,$x_2,...,x_n,x_{n+1}\in X$. 
 
\end{itemize}
Then the function $G_n$ is called a \emph{Generalized $n$-metric} on $X$, and the pair $(X,G_n)$ a $ \emph{Generalized $n$-metric space}$.
\end{definition} 
From now on we always have $n\ge 3$ for $(X,G_n)$ to be a generalized $n$-metric space.

\begin{example} Define a function $\rho\colon \mathbb{R}^n \to \mathbb{R}^+$,$(n\ge3)$ by
\begin{equation*}
\rho(x_1,x_2,\dots ,x_n)= \text{max} \{ \left|x_r-x_s\right|\colon r,s\in \{1,2,...n\},r\neq s \}
\end{equation*}
for all $x_1$, $x_2,...,x_n \in X$. Then $( \mathbb{R}, \rho)$ is a generalized $n$-metric space.
\end{example}

\begin{example}
\label{exmpl} For any metric space $(X,d)$, the following functions define generalized $n$-metrics on $X$:
\begin{itemize}
\item [(1)] $K_1^d(x_1,x_2,...,x_n)=\sum_{r} \sum_{s}d(x_r,x_s)$,

\item [(2)] $K_2^d(x_1,x_2,...,x_n)=\text{max}\{d(x_r,x_s)\colon r,s\in \{1,2,...,n\},r\neq s\}$.
\end{itemize}
\end{example}
\begin{definition} \cite{KHA2} A $G_n$-metric space $(X,G_n)$ is called symmetric if
\begin{equation}\label{symm}
G_n(x,y,y,...,y)=G_n(x,x,x,...,y)
\end{equation}
\end{definition}
\begin{proposition}\cite{KHA2}
\label{inequalty}\label{ineq5}
Let $G_n\colon X^n \to \mathbb {R}^+$,$(n\ge 3)$ be a generalized $n$-metric defined on $X$, then for $x$,$y$ $\in X$ we have\\
\begin{equation}
\label{ineq1}
G_n(x,y,y,\dots ,y)\leq (n-1) G_n(y,x,x,\dots ,x)
\end{equation}
\end{proposition}
\begin{definition}\cite{KHA2} Let $(X,G_n)$ be a generalized $n$-metric space, then for $x_0\in X$,$r>0$, the \emph{$G_n$-ball} with centre $x_0$ and radius $r$ is
\begin{equation*}
B_G(x_0,r)=\{ y\in X \colon G_n(x_0,y,y,\dots ,y)<r \}
\end{equation*}

\end{definition}
\begin{proposition}\cite{KHA2}  Let $(X,G_n)$ be a generalized $n$-metric space, then the $G_n$-ball is open in $X$.
\end{proposition}

Hence the collection of all such balls in $X$ is closed under arbitrary union and finite intersection and therefore induces a topology on $X$ called the generalized $n$-metric topology $\Im (G_n)$ generated by the generalized $n$-metric on $X$.\\
From example~\ref{exmpl} it is clear that for a given metric we can always define generalized $n$-metrics. The converse is also true for if $G_n$ is a generalized $n$-metric then we can define a metric $d_G$ as follows-
\begin{equation} \label{met}
d_G(x,y)=G_n(x,y,y,\dots ,y)+G_n(x,x,\dots ,x,y)
\end{equation}

\begin{proposition}\cite{KHA2} Let $B_{d_G}(x,r)$ denote the open ball in the metric space $(X,d_G)$ and $B_G(x,r)$ the $G_n$-ball in the corresponding generalized $n$-metric space $(X,G_n)$. Then we have
\begin{equation*}
B_G(x,\frac{r}{n})\subseteq B_{d_G}(x,r)
\end{equation*}
\end{proposition}

This indicates that the topology induced by the generalized $n$-metric on $X$ coincides with the metric topology induced by the metric $d_G$. Thus every generalized $n$-metric space is topologically equivalent to a metric space.

\begin{definition}\cite{KHA2} Let $(X,G_n)$ be a generalized $n$-metric space. A sequence $<x_m>$ in $X$ is said to be \emph{$G_n$-convergent} if it converges to a point $x$ in the generalized $n$-metric topology $\Im (G_n)$ generated by the $G_n$-metric on $X$.
\end{definition}
\begin{proposition}\cite{KHA2} 
\label{convsq} Let $G_r\colon X^r \to \mathbb{R}^+$, $(r\ge 3)$ be a generalized $r$-metric defined on $X$. Then for a sequence $<x_n>$ in $X$ and $x\in X$ the following are equivalent:
\begin{itemize}
\item [(1)] The sequence $<x_n>$ is $G_r$-convergent to $x$.
\item [(2)] $d_G(x_n,x)\rightarrow 0$ as $n\rightarrow \infty$.
\item [(3)] $G_r(x_n,x_n,...,x_n,x)\rightarrow 0$ as $n\rightarrow \infty$.
\item [(4)] $G_r(x_n,x,...,x)\rightarrow 0$ as $n\rightarrow \infty$.
\end{itemize}
\end{proposition}

\begin{definition}\cite{KHA2}
Let $(X,G^X_n)$ and $(Y,G^Y_n)$ be generalized $n$-metric spaces. A function $f\colon X \to Y$ is said to be \emph{Generalized $n$-continuous} or $G_n$-Continuous at a point $x\in X$ if $f^{-1}(B_{G^Y_n}(f(x),r))\in \Im(G^X_n)$, for all $r>0$. The function $f$ is said to be generalized $n$-continuous if it is generalized $n$-continuous at all points of $X$.
\end{definition}
Since every generalized $n$-metric space is topologically equivalent to a metric space, hence we have the following result:
\begin{proposition}\cite{KHA2} 
\label{cont}
 Let $(X,G^X_n)$ and $(Y,G^Y_n)$ be generalized $n$-metric spaces. A function $f\colon X \to Y$ is said to be generalized $n$-continuous or $G_n$-Continuous at a point $x\in X$ if and only if it is generalized $n$-sequentially continuous at $x$; that is, whenever the sequence $<x_m>$ is $G^X_n$-convergent to $x$, the sequence $<f(x_m)>$ is $G^Y_n$-convergent to $f(x)$.
\end{proposition}

\begin{proposition}\cite{KHA2} Let $(X,G_n)$ be a generalized $n$-metric space, then the function $G_n(x_1,x_2,...,x_n)$ is jointly continuous in the variables $x_1$,$x_2,...,x_n$.
\end{proposition}

\begin{definition}\cite{KHA2} 
\label{cauchy} Let $(X,G_m)$ be a generalized $m$-metric space. A sequence $<x_n>$ in $X$ is said to be $G_m$-Cauchy if for every $\epsilon >0$, there exists $N\in \mathbb{N}$ such that
\begin{equation*}
G_m(x_{n_1},x_{n_2},...,x_{n_m})< \epsilon \; \text{for all} \; n_1,n_2,...,n_m\ge N
\end{equation*}
\end{definition}
\begin{proposition}\cite{KHA2}
Let $(X,G_m)$ be a generalized $m$-metric space. A sequence $<x_n>$ in $X$ is $G_m$-Cauchy if and only if for every $\epsilon >0$, there exists $N\in \mathbb{N}$ such that
\begin{equation}
\label{condcchy}
G_m(x_{n_1},x_{n_2},...,x_{n_2}) <\epsilon \; \text{for all} \; n_1,n_2\ge N
\end{equation}
\end{proposition}

\begin{proposition}\cite{KHA2}
Every $G_n$-convergent sequence in a generalized $n$-metric space is $G_n$-Cauchy.
\end{proposition}

\begin{definition}\cite{KHA2} A generalized $n$-metric space $(X,G_n)$ is said to be \emph{$G_n$-complete} if every $G_n$-Cauchy sequence in $(X,G_n)$ is $G_n$-convergent in $(X,G_n)$.
\end{definition}

\begin{definition} Let $(X,d)$ be a metric space. A mapping $T:X\to X$ is called a \textit{contraction} if there exists $r\in [0,1)$ such that $d(Tx,Ty)\le rd(x,y)$, for all $x,y\in X$. 
\end{definition}
\begin{theorem} \cite{BAN} (Banach Contraction Principle) If $(X,d)$ is a complete metric space, then every contraction $T$ on $X$ has a unique fixed point.
\end{theorem}

Many fixed point theorems have been proved as generalizations of Banach fixed point theorem.   The following remarkable generalization is due to Suzuki \cite{SUZ}.
\begin{theorem} ( Suzuki \cite{SUZ} ) \label{thmsuz}
 Let $(X,d)$ be a complete metric space and let $T$ be a mapping on $X$. Define a nonincreasing function $\theta:[0,1)\to (1/2,1]$ by
 
 \begin{equation}\label{condsuz1}
 \theta(r)=
 \begin{cases}
  1 \qquad \qquad  &\text{if } 0\le r \le (\sqrt{5}-1)/2,\\
  
  (1-r)r^{-2} \qquad  &\text{if } (\sqrt{5}-1)/2\le r\le 2^{-1/2},\\
  (1+r)^{-1} \qquad  &\text{if } 2^{-1/2}\le r<1.
  \end{cases}
  \end{equation}
  Assume that there exists $r\in [0,1)$ such that 
  $\theta(r)d(x,Tx)\le d(x,y)$ implies $d(Tx,Ty)\le rd(x,y)$ for all $x,y\in X$. Then there exists a unique fixed point $z$ of $T$. Moreover $\lim_n {T^n}x=z$ for all $x\in X$.
\end{theorem}
\section{main results}

First we prove the Banach fixed point theorem in the framework of generalized n-metric space.

\begin{theorem} \label{banach}
Let $(X,G_r)$ be a complete generalized $r$-metric space and let $T:X\to X$ be a mapping satisfying the following condition for all $x_1,x_2,\dots x_r\in X$
\begin{equation} \label{eqban}
G_r(Tx_1,Tx_2,\dots,Tx_r)\le kG_r(x_1,x_2,\dots, x_r)
\end{equation}
where $k\in [0,1)$. Then $T$ has a unique fixed point.
\end{theorem}
\begin{proof}
Let $y_0$ be an arbitrary point in $X$. Consider a sequence $<y_n>$ in $X$ such that $y_n=T^n y_0$.

using the condition ~\ref{eqban} we have 

\begin{align*}
\begin{split}
G_r(Ty_{n-1},Ty_n,\dots,Ty_n)&\le kG_r(y_{n-1},y_n,\dots, y_n)\\
\text{or} \quad G_r(y_n,y_{n+1},\dots,y_{n+1})&\le  kG_r(y_{n-1},y_n,\dots, y_n)
\end{split}
\end{align*} 
By the repeated application of condition ~\ref{eqban}, we have
\begin{equation*}
G_r(y_n,y_{n+1},\dots,y_{n+1})\le  k^nG_r(y_0,y_1,\dots, y_1)
\end{equation*}

Now we claim that the sequence $<y_n>$ in $X$ is $G_r$-Cauchy sequence in $X$.\\
For all natural numbers $n$ and $m(>n)$ we have from [G 5]
\begin{align*}
\begin{split}
G_r(y_n,y_m,\dots,y_m)&\le G_r(y_n,y_{n+1},\dots,y_{n+1})+G_r(y_{n+1},y_{n+2},\dots,y_{n+2})+\dots\\
&\qquad \cdots +G_r(y_{m-1},y_m,\dots ,y_m)\\
&\le (k^n+k^{n+1}+\dots +k^{m-1})\,G_r(y_0,y_1,\dots ,y_1)\\
&\le (k^n+k^{n+1}+\dots)\,G_r(y_0,y_1,\dots ,y_1)\\
&=\frac{k^n}{1-k}\,G_r(y_0,y_1,\dots ,y_1)\rightarrow 0 \, \text{as} \, n,m \rightarrow \infty
\end{split}
\end{align*}
Hence the sequence $<y_n>$ is a $G_r$-Cauchy sequence in $X$. By completeness of $(X,G_r)$, there exists a point $u\in X$ such that $<y_n>$ is $G_r$-convergent to $u$.
Suppose that $Tu\ne u$, then 
\begin{align*}
\begin{split}
G_r(Tu,\dots,Tu,y_n)&= G_r(Tu,\dots,Tu,Ty_{n-1})\\
&\le k\,G_r(u,\dots ,u,y_{n-1})\\
\end{split}
\end{align*}
Taking the limit as $n\to \infty$, we have
\begin{align*}
\begin{split}
G_r(Tu,\dots,Tu,u)&\le k G_r(u,\dots,u,u)=0\\
\text{or}  \qquad  G_r(Tu,\dots,Tu,u)&\le 0\\
\end{split}
\end{align*}
But from [G 2] $G_r(Tu,\dots,Tu,u)>0$. Thus we get a contradiction. Hence we have $u=Tu$. For uniqueness of $u$, suppose that $v\neq u$ is such that $Tv=v$. Then we have
\begin{align*}
\begin{split}
G_r(u,v,\dots ,v)=G_r(Tu,Tv, \dots ,Tv)\le k\,G_r(u,v,\dots ,v)<G_r(u,v,\dots ,v)
\end{split}
\end{align*} 
Since $k\in [0,1)$. Thus we get a contradiction, hence we have $u=v$.

\end{proof}
The theorem ~\ref{thmsuz} initiated a lot of research work in the form of various variations, refinements and generalizations. We shall now prove a similar theorem in generalized $n$-metric spaces.
\begin{theorem} \label{suz}
 Let $(X,G_n)$ be a complete $G_n$-metric space and let $T$ be a mapping on $X$. Define a strictly decreasing function $\theta$ from $[0,1)$ onto $(1/2,1]$ by $\theta(r)=\frac{1}{1+r}$. Assume that there exists $r\in[0,1)$ such that for every $u,v\in X$, the inequality
\begin{align}\label{condsuz2}
\begin{split}
&\theta(r)G_n(u,Tu,\dots ,Tu)\le G_n(u,v,\dots,v)\\
\text{implies}\qquad&\;G_n(Tu,Tv,\dots,Tv)\le rG_n(u,v,\dots,v)
\end{split}
\end{align}
Then there exists a unique fixed point $y$ of $T$, i.e. $Ty=y$. Moreover $T$ is $G_n$-Continuous at $y$.
\end{theorem} 
\begin{proof} Let us first assume that $(X,G_n)$ is symmetric, i.e. condition (equation) ~\ref{symm} holds. Then from relation ~\ref{met} we have
\begin{equation}
d_G(x,y)=2G_n(x,y,...,y)
\end{equation}
condition ~\ref{condsuz2} gives
\begin{equation}
\theta(r)d_G(x,Tx)\le d_G(x,y)\; \text{implies} \; d_G(Tx,Ty)\le rd_G(x,y)
\end{equation} 
Then the metric space $(X,d_G)$ satisfies the conditions of Theorem ~\ref{thmsuz} with $\theta(r)=\frac{1}{1+r}$ the required decreasing function. Therefore from Theorem~\ref{thmsuz}, $T$ has a unique fixed point.

Now suppose that $(X,G_n)$ is not symmetric. Since $\theta(r)\le 1$, We have \\
$\; \; \theta(r)G_n(x,Tx,...,Tx)
\le G_n(x,Tx,...,Tx)$ for all $x\in X$. Hence by condition~\ref{condsuz2} of the theorem, this implies that for all $x\in X$, we have
\begin{equation}
G_n(Tx,T^2x,...,T^2x)\le rG_n(x,Tx,...,Tx)
\end{equation}
Let $y_0\in X$. Define a sequence $<y_m>$ in $X$ such that $y_n=T^mx_0$. Then We have
\begin{align}
\begin{split}
G_n(y_m, y_{m+1},...,y_{m+1})&= G_n(T^my_0,T^{m+1}y_0,...,T^{m+1}y_0)\\
&\le rG_n(T^{m-1}y_0,T^my_0,...,T^my_0)\\
&\vdots \\
&\le r^mG_n(y_0,Ty_0,...Ty_0)
\end{split}
\end{align}
Following the proof of Theorem ~\ref{banach}, we can show that the sequence $<y_m>$ in $X$ is $G_n$-Cauchy in $X$. By completeness of $(X,G_n)$, there exists a point $y\in X$ such that $<y_m>$ is $G_n$-convergent to $y$. Thus there exists a natural number $k$ and $h>1$ such that
for all $m\ge k$, $x(\ne y)\in X$ we have
\begin{align*}
\begin{split}
G_n(y_m,y,y,...,y)&\le \frac{1}{h}G_n(x,y,...,y)\\
\text{and} \quad G_n(y_m,y_m,...,y_m,y)&\le \frac{1}{h}G_n(x,y,...,y) 
\end{split}
\end{align*} 
Then We have 
\begin{align*}
\begin{split}
\theta(r)G_n(y_m,Ty_m,...,Ty_m)&\le \frac{1}{h}G_n(y_m,Ty_m,...,Ty_m)\\
&=G_n(y_m,y_{m+1},...,y_{m+1})\\
& \le G_n(y_m,y,...,y)+G_n(y,y_{m+1},....,y_{m+1})\\
&\le G_n(y_n,y,...,y)+(n-1)G_n(y_{m+1},...,y_{m+1},y)\\
&\le \frac{1}{h}G_n(x,y,...,y)+\frac{(n-1)}{h}G_n(x,y,...,y)\\
&=\frac{n}{h}G_n(x,y,...,y)\\
&= \frac{n}{h-1}\Big[G_n(x,y,...,y)-\frac{1}{h}G_n(x,y,...,y)\Big]\\
&\le \frac{n}{h-1}\Big[G_n(x,y,...,y)-G_n(y_m,y,...,y)\Big]\\
&\le \frac{n}{h-1}G_n(x,y_m,...,y_m)\\
&\le \frac{n}{h-1}(n-1)G_n(y_m,x,...,x)\quad \text{by proposition ~\ref{ineq5}}\\
\end{split}
\end{align*}
If we choose $h>n^2-n+1$, then we have
\begin{equation*}
\theta(r)G_n(y_m,Ty_m,...,Ty_m)<G_n(y_m,x,...,x)
\end{equation*}
Hence by hypothesis (relation ~\ref{condsuz2}), We have 
\begin{align*}
\begin{split}
G_n(Ty_m,Tx,...,Tx)&\le rG_n(y_m,x,...,x)\\
\text{or}\quad G_n(y_{m+1},Tx,...,Tx)&\le rG_n(y_m,x,..,x)\; \text{for all}\;m\ge k
\end{split}
\end{align*}
Making $m\to \infty$, We have
\begin{equation*}
G_n(y,Tx,...,Tx)\le rG_n(y,x,...,x)\; \text{for all} \; x\in X\; \text{with}\;x\ne y
\end{equation*}
We now prove that $y$ is a fixed point of $T$. 

On the contrary, suppose that $Ty\ne y$. 
We claim that 
\begin{align*}
\begin{split}
\text{either}\;&\theta(r)G_n(x,Tx,...,Tx)\le G_n(x,z,...,z)\\
\text{or}\; & \theta(r)G_n(Tx,T^2x,...,T^2x)\le G_n(Tx,z,...,z)\; \text{for every}\;x,z\in X.
\end{split}
\end{align*}
or in light of inequality~\ref{ineq1} we have
\begin{align*}
\begin{split}
\text{either}\;&\theta(r)G_n(x,Tx,...,Tx)\le G_n(x,z,...,z)\\
\text{or}\; & \theta(r)G_n(Tx,T^2x,...,T^2x)\le (n-1)G_n(z,Tx,...,Tx)\; \text{for every}\;x,z\in X.\\
\text{i.e., either}\;&\theta(r)G_n(x,Tx,...,Tx)\le G_n(x,z,...,z)\\
\text{or}\;  \frac{1}{n-1}& \theta(r)G_n(Tx,T^2x,...,T^2x)\le G_n(z,Tx,...,Tx)\; \text{for every}\;x,z\in X.
\end{split}
\end{align*}

For if $\theta(r)G_n(x,Tx,...,Tx)>G_n(x,z,...,z)$ or $\frac{1}{n-1}\theta(r)G_n(Tx,T^2x,...,T^2x) >G_n(z,Tx,...,Tx)$. Then we have 
\begin{align*}
\begin{split}
G_n(x,Tx,...,Tx)&\le G_n(x,z,...,z)+G_n(z,Tx,...,Tx)\\
&<\theta(r)G_n(x,Tx,...,Tx)+\frac{1}{n-1}\theta(r)G_n(Tx,T^2x,...,T^2x)\\
&=\theta(r)\Big[G_n(x,Tx,...,Tx)+\frac{1}{n-1}G_n(Tx,T^2x,...,T^2x)\Big ]\\
&\le \theta(r)\Big[G_n(x,Tx,...,Tx)+\frac{1}{n-1}rG_n(x,Tx,...,Tx)\Big]\\
&=\theta(r)G_n(x,Tx,...,Tx)\Big[1+\frac{r}{n-1}\Big]\\
&=\Bigg(\frac{1+\frac{r}{n-1}}{1+r}\Bigg)G_n(x,Tx,...,Tx)
\end{split}
\end{align*}
Since we have $n\ge 3$, therefore we get $G_n(x,Tx,...,Tx)<G_n(x,Tx,...,Tx)$, a contradiction. Thus our claim that for $x,z\in X$,
\begin{align*}
\begin{split}
\text{either}\;&\theta(r)G_n(x,Tx,...,Tx)\le G_n(x,z,...,z)\\
\text{or}\; & \theta(r)G_n(Tx,T^2x,...,T^2x)\le G_n(Tx,z,...,z)\; \text{is true}.
\end{split}
\end{align*}
This implies that either
\begin{align*}
\begin{split}
&\theta(r)G_n(y_{2m},Ty_{2m},...,Ty_{2m})\le G_n(y_{2m},y,...,y)\\
\text{or}\; & \theta(r)G_n(y_{2m+1},Ty_{2m+1},...,Ty_{2m+1})\le G_n(y_{2m+1},y,...,y)\; \text{for every}\;m\in \mathbb{N}.
\end{split}
\end{align*} 
Therefore the condition~\ref{condsuz2} of the theorem implies that either
\begin{align*}
\begin{split}
&G_n(y_{2m+1},Ty,...,Ty)\le rG_n(y_{2m},y,...,y)\\
\text{or}\; & G_n(y_{2m+2},Ty,...,Ty)\le r G_n(y_{2m+1},y,...,y)\; \text{ holds for every}\;m\in \mathbb{N}.
\end{split}
\end{align*} 
Now $y_m\to y$, the above inequalities imply that there exists a subsequence of the sequence $<y_m>$ which converges to $Ty$. Thus we have $Ty=y$ contradicting our initial assumption. Hence $Ty=y$.

For uniqueness of $y$, suppose that $u\neq y$ is such that $Tu=u$. Then we have $G_n(y,u,\dots ,u)>0$ and $\theta(r)G_n(y,Ty,...,Ty)=0$ satisfying the condition 
$\theta(r)G_n(y,Ty,...,Ty)\le G_n(y,u,\dots ,u)$. By using condition ~\ref{condsuz2}, we get
\begin{equation*}
G_n(y,u,...,u)=G_n(Ty,Tu,...,Tu)\le rG_n(y,u,..,u)<G_n(y,u,...,u)
\end{equation*}
Thus we get a contradiction, hence we have $y=u$.
To prove the $G_n$-continuity (i.e. generalized $n$-continuity) of $T$ at $y$, We use the proposition ~\ref{cont}. Consider any sequence $<u_m>$ converging (i.e. $G_n$-convergent) to $y\in X$. Then we have 
\begin{equation*}
\theta(r)G_n(y,Ty,...,Ty)=0\le G_n(y,u_m,...,u_m)
\end{equation*}
Which implies that 
\begin{equation*}
G_n(Ty,Tu_m,...,Tu_m)\le rG_n(y,u_m,...,u_m)
\end{equation*}
i.e., $G_n(y,Tu_m,...,Tu_m)\le rG_n(y,u_m,...,u_m)$. 

Making $m\to \infty$, We get
\begin{equation*}
\lim_{m\to \infty} G_n(y,Tu_m,...,Tu_m)=0.
\end{equation*}
Hence $Tu_m\to y $, i.e. the sequence $<Tu_m>$ is Generalized $n$-convergent to $y(=Ty)$. Therefore by proposition ~\ref{cont}, the mapping $T$ is $G_n$-Continuous at $y$.
\end{proof}
\section{Application to functional equations}
Some functional equations arise in multistage decision processes where the origin of the theory of dynamic programming lies(\cite{BELL1},\cite{BELL2}). The existence and uniqueness of the solutions of these functional equations have been studied by several authors (\cite{BHA1},\cite{KAL1},\cite{LIU},\cite{SAL},\cite{SIN}) using fixed point theorems. In this section, we study the existence of solution of one such functional equation using theorem ~\ref{suz}.\par
\par Suppose that $U$ and $V$ are Banach spaces. Let $S\subset U$ be the state space and $D\subset V$ be the decision space. Let us denote a state vector by $x$ and a decision vector by $y$. Let $g\colon S\times D \to \mathbb{R}$, $M\colon S\times D\times \mathbb{R} \to \mathbb{R}$ be the bounded functions and $\tau \colon S\times D \to S$ be the transformation of decision process.\par 
The return function $f\colon S\to \mathbb{R}$ of the continuous decision process is defined by the functional equation
\begin{equation}\label{funceq}
f(x)=\sup_{y\in D}\big[g(x,y)+M\big(x,y,f(\tau (x,y))\big)\big],\,\, x\in S
\end{equation}
Let $B(S)$ be the set of all real valued bounded functions on $S$. For $\psi,\phi \in B(S)$, let
\begin{equation*}
d(\psi,\phi)=\sup\{ | \psi(x)-\phi(x)|:x\in S\}
\end{equation*}
Obviously $d$ is a metric on $B(S)$ and $(B(S),d)$ is a complete metric space. Let us denote $B(S)$ by $X$. If we define $G_n:X^n\to \mathbb{R}^+\,(n\ge 3)$  by
\begin{equation*}
G_n(\psi_1,\psi_2,\dots ,\psi_n)=\max \{d(\psi_p,\psi_q):1\le p<q\le n\}
\end{equation*}
Then $(X,G_n)$ is a $G_n$-complete generalized $n$-metric space. Let $\theta$ be the function as defined in theorem ~\ref{suz} and $T\colon X \to X$ be the mapping defined by
\begin{equation}\label{funceq1}
T(\psi(x))=\sup_{y\in D}\big[g(x,y)+M\big(x,y,\psi(\tau (x,y))\big)\big],\,\, x\in S, \psi \in X
\end{equation}
Then the existence and uniqueness of the solution of functional equation ~\ref{funceq} are established by the following result:
\begin{theorem}
Suppose that there exists $r\in [0,1)$ such that for every $(x,y) \in S\times D$, $\psi, \phi \in X$ and $t\in S$, the inequality
\begin{equation}\label{condfunc}
\theta(r)G_n(\psi, T\psi, \dots , T\psi)\le G_n(\psi,\phi,\dots, \phi)
\end{equation}
implies
\begin{equation}
|M(x,y,\psi(t))-M(x,y,\phi(t))|\le r|\psi(t)-\phi(t)|
\end{equation}
Then the functional equation ~\ref{funceq} has a unique bounded solution in $X$.
\end{theorem}
\begin{proof}
Let $\lambda$ be an arbitrary positive real number and $\psi,\phi\in X$. For $x\in S$, let us choose $y_1,y_2\in D$ such that
\begin{equation}\label{cond1}
 T(\psi(x))<g(x,y_1)+M\big(x,y_1,\psi(\tau (x,y_1))\big)+\lambda
 \end{equation}
 \begin{equation}\label{cond2}
 T(\phi(x))<g(x,y_2)+M\big(x,y_2,\phi(\tau (x,y_2))\big)+\lambda
 \end{equation}
 By the definition of mapping $T$ and equation ~\ref{funceq1}, we have
 \begin{equation}\label{cond3}
  T(\psi(x))<g(x,y_2)+M\big(x,y_2,\psi(\tau (x,y_2))\big)
  \end{equation}
  \begin{equation}\label{cond4}
  T(\phi(x))<g(x,y_1)+M\big(x,y_1,\phi(\tau (x,y_1))\big)
  \end{equation}
  If the inequality ~\ref{condfunc} holds, then from inequalities ~\ref{cond1} and ~\ref{cond4}, we have
  \begin{align*}
  \begin{split}
   T(\psi(x))-T(\phi(x))&<M\big(x,y_1,\psi(\tau (x,y_1))\big)-M\big(x,y_1,\phi(\tau (x,y_1))\big)+\lambda\\
   &\le |M\big(x,y_1,\psi(\tau (x,y_1))\big)-M\big(x,y_1,\phi(\tau (x,y_1))\big)|+\lambda    
  \end{split}  
  \end{align*}
  Let $\tau(x,y_1)=x_1\in S$, then 
  \begin{align*}
  \begin{split}
   T(\psi(x))-T(\phi(x))&<|M\big(x,y_1,\psi(x_1))\big)-M\big(x,y_1,\phi(x_1)\big)|+\lambda\\
   &\le r|\psi(x_1)-\phi(x_1)|+\lambda\\
   &\le r\, G_n(\psi,\phi,\dots,\phi)+\lambda  
  \end{split}
  \end{align*}
  \begin{equation}\label{cond5}
  \text{or}\; \; T(\psi(x))-T(\phi(x))<r\, G_n(\psi,\phi,\dots,\phi)+\lambda
  \end{equation}
  
  Similarly from inequalities ~\ref{cond2} and ~\ref{cond3}, we have
  \begin{equation}\label{cond6}
   T(\phi(x))-T(\psi(x))<r\,G_n(\psi,\phi,\dots,\phi)+\lambda 
 \end{equation} 
Hence from inequalities ~\ref{cond5} and ~\ref{cond6}, we have
 \begin{equation}\label{cond7}
   |T(\psi(x))-T(\phi(x))|<r\,G_n(\psi,\phi,\dots,\phi)+\lambda 
 \end{equation} 
 The inequality ~\ref{cond7} is true for every $x\in S$, hence we have 
 \begin{equation*}
   G_n(T\psi,T\phi,\dots,T\phi)\le r\,G_n(\psi,\phi,\dots,\phi)+\lambda 
  \end{equation*} 
  Since $\lambda>0$ is arbitrary, hence
   \begin{equation*}
     G_n(T\psi,T\phi,\dots,T\phi)\le r\,G_n(\psi,\phi,\dots,\phi)
    \end{equation*}
 Therefore the inequality
\begin{equation*}
 \theta(r)G_n(\psi, T\psi, \dots , T\psi)\le G_n(\psi,\phi,\dots, \phi)
\end{equation*}
 implies
\begin{equation*}
      G_n(T\psi,T\phi,\dots,T\phi)\le r\,G_n(\psi,\phi,\dots,\phi)
\end{equation*}
Thus all the conditions of the theorem ~\ref{suz} are satisfied and hence the functional equation ~\ref{funceq} has a unique bounded solution.
\end{proof}

\end{document}